\documentclass[11pt]{article}
\usepackage{graphicx}
\usepackage{amssymb}
\usepackage{amsmath}
\usepackage{amsthm}
\usepackage{amsfonts}
\usepackage{url}
\usepackage{hyperref}
\usepackage{breakurl}
\usepackage[T1]{fontenc}        
\usepackage[usenames]{color}

\newcommand{\comment}[1]{}
\newcommand{\raisecomma}{\raisebox{2pt}{$,$}}
\newcommand{\raisedot}{\raisebox{2pt}{$.$}}
\newcommand{\sign}{\text{sgn}}

\newcommand{\Dbar}{{\mathcal D}} 
\newcommand{\Rbar}{{\mathcal R}}
\newcommand{\Had}{{\mathcal H}}

\newcommand{\R}{{\mathbb R}}

\newcommand{\Z}{{\mathbb Z}}

\newcommand{\Prob}{{\mathbb P}} 
\newcommand{\E}{{\mathbb E}}    
\newcommand{\V}{{\mathbb V}}    

\newcommand{\ve}{\varepsilon}

\newcommand{\Tablea}{{$1$}}	
\newcommand{\Tableb}{{$2$}}	

\newtheorem{theorem}{Theorem}[section]

\newtheorem{lemma}[theorem]{Lemma}

\newtheorem{remark}[theorem]{Remark}

\begin{document}
\bibliographystyle{plain}
\title{~\\[-40pt]
Probabilistic lower bounds on maximal determinants of binary matrices
}
\author{
\sc Richard P.\ Brent\\
\small\it Australian National University\\
\small\it Canberra, ACT 2600\\
\small\it Australia \\
\and
\sc Judy-anne H.\ Osborn\\
\small\it The University of Newcastle\\
\small\it Callaghan, NSW 2308\\
\small\it Australia\\
\and
\sc Warren D.\ Smith\\
\small\it Center for Range Voting\\
\small\it 21 Shore Oaks Drive\\ 
\small\it Stony Brook, NY 11790\\
\small\it USA\\
}

\date{
\small\it In memory of Mirka Miller 1949--2016} 

\maketitle
\thispagestyle{empty}                   

\begin{abstract}
Let $\Dbar(n)$ be the maximal determinant for $n \times n$ $\{\pm 1\}$-matrices,
and $\Rbar(n) = \Dbar(n)/n^{n/2}$ be the ratio of $\Dbar(n)$ to the Hadamard
upper bound.
Using the probabilistic method, we prove
new lower bounds on $\Dbar(n)$ and $\Rbar(n)$ in terms of 
$d = n-h$, where $h$ is the order of a Hadamard matrix
and $h$ is maximal subject to $h \le n$. 
For example,
\[
\Rbar(n) > \left(\frac{2}{\pi e}\right)^{d/2}
 \;\text{ if }\; 1 \le d \le 3,\;\text{ and}
\]
\[
\Rbar(n) > \left(\frac{2}{\pi e}\right)^{d/2}
	\left(1 - d^2\left(\frac{\pi}{2h}\right)^{1/2}\right)
\;\text{ if }\; d > 3.
\]
By a recent result of Livinskyi,
$d^2/h^{1/2} \to 0$ as $n \to \infty$, so the second bound
is close to $(\pi e/2)^{-d/2}$ for large~$n$.
Previous lower bounds
tended to zero as $n \to \infty$ with $d$ fixed, except in the cases
$d \in \{0,1\}$.  For $d \ge 2$, our bounds are better for all
sufficiently large~$n$.  If the Hadamard conjecture is true,
then $d \le 3$, so the first bound above
shows that $\Rbar(n)$ is bounded below by a positive
constant $(\pi e/2)^{-3/2} > 0.1133$.
\end{abstract}

\pagebreak[3]

\section{Introduction}		\label{sec:intro}

Let $\Dbar(n)$ be the maximal determinant possible for an $n\times n$
matrix with elements in $\{\pm1\}$.
Hadamard~\cite{Hadamard} 
proved that $\Dbar(n) \le n^{n/2}$, and the
\emph{Hadamard conjecture} is that a matrix achieving this upper bound
exists for each positive integer $n$ divisible by four.
The function $\Rbar(n) := \Dbar(n)/n^{n/2}$ 
is a measure of the sharpness of the Hadamard bound.
Clearly $\Rbar(n) = 1$ if a Hadamard matrix of order $n$
exists; otherwise $\Rbar(n) < 1$.  In this paper we
give lower bounds on $\Dbar(n)$ and $\Rbar(n)$.

Let $\Had$ be the set of orders of Hadamard matrices, and let
$h\in\Had$ be maximal subject to $h \le n$.
Then $d = n-h$ can be regarded as the ``gap'' between $n$ and the nearest
(lower) Hadamard order. We are interested the case that $n$ is not a 
Hadamard order, i.e.\ $d > 0$ and $\Rbar(n) < 1$.

Except in the cases $d \in \{0,1\}$, previous lower bounds on $\Rbar(n)$
tended to zero as $n \to \infty$.
For example, the well-known bound of
Clements and Lindstr\"om~\cite[Corollary to Thm.~2]{CL65} shows that
$\Rbar(n) > (3/4)^{n/2}$, and \hbox{\cite[Thm.~9]{rpb249}} shows that
$\Rbar(n) \ge (ne/4)^{-d/2}$. 
In contrast, our results imply that, for fixed~$d$,
$\Rbar(n)$ is bounded below by a positive constant (depending only on~$d$).

Our lower bound proof uses the probabilistic method
pioneered by Erd\H{o}s
(see for example~\cite{AS,ES}). 
This method does
not appear to have been applied previously to the Hadamard maximal
determinant problem, except in the case $d = 1$ (so $n \equiv 1
\bmod 4$); in this case the concept of \emph{excess} has been used~\cite{FK}, 
and lower bounds on the maximal excess were obtained by the probabilistic
method~\cite{Best,BS,ES,FK}.  

\S\ref{sec:construction} describes our probabilistic construction
and determines
the mean $\mu$ and variance $\sigma^2$ 
of elements in the Schur complement generated by the construction
(see Lemmas~\ref{lemma:variance} and~\ref{lemma:sigma_asymptotics}).
Informally, 
we adjoin $d$ extra columns to an $h \times h$ Hadamard matrix $A$,
and fill their $h \times d$ entries with random 
(uniformly and independently distributed) $\pm 1$ values.
Then we adjoin $d$ extra rows, and fill their $d\times(h+d)$ entries with
values chosen deterministically 
in a way intended to approximately 
maximise the determinant of the final matrix $\widetilde{A}$.
To do so, we use the fact 
that this 
determinant can be expressed in terms 
of the  $d \times d$ Schur complement 
of $A$ in $\widetilde{A}$. 

In the case $d=1$, this method is essentially the same as the known method
involving the {excess} of matrices Hadamard-equivalent to $A$,
and leads to the same bounds that can be obtained by bounding the excess in
a probabilistic manner.

In \S\ref{sec:bounds}
we give lower bound results on both $\Dbar(n)$ and $\Rbar(n)$.
Of course, a lower bound on $\Dbar(n)$ immediately gives an equivalent
lower bound on $\Rbar(n)$.  However, we use some elementary inequalities to
obtain simpler (though slightly weaker)
bounds on $\Rbar(n)$.  For example, if $d\le 3$ then 
Theorem~\ref{thm:small_d} states that $\Dbar(n) \ge h^{h/2}(\mu^d-\eta)$,
where $\mu$ and $\eta$ are certain functions of $h$ and $d$.  
Theorem~\ref{thm:small_d} also states the (weaker) result
that $\Rbar(n) > (\pi e/2)^{-d/2}$. The lower bound on $\Rbar(n)$ 
clearly shows that the
ratio of our bound to the Hadamard bound is at least
$(\pi e/2)^{-3/2} > 0.1133$, whereas this conclusion
is not immediately obvious from the lower bound on $\Dbar(n)$.

We outline the bounds on $\Rbar(n)$ here.
Theorem~\ref{thm:lower_bd_via_Chebyshev} gives a lower
bound
\begin{equation}				\label{eq:rnbnd1}
\Rbar(n) > \left(\frac{2}{\pi e}\right)^{d/2}
	\left(1 - d^2\left(\frac{\pi}{2h}\right)^{1/2}\right)
\end{equation}
which is nontrivial
whenever $h > \pi d^4/2$. By the results of Livinskyi~\cite{Livinskyi},
$d = O(h^{1/6})$ as $h \to \infty$ (see~\cite[\S6]{rpb257} for details),
so the condition $h > \pi d^4/2$ holds
for  all sufficiently large~$n$.
Also, as $n\to\infty$, $d^2/h^{1/2} = O(n^{-1/6}) \to 0$, so the
lower bound~\eqref{eq:rnbnd1} is close to $(\pi e/2)^{-d/2}$.
For fixed $d > 1$ and large~$n$,
our lower bounds on $\Rbar(n)$ are better
than previous bounds
(see Table~{\Tablea} in~\S\ref{sec:numerics}).

Theorem~\ref{thm:small_d} applies only for $d \le 3$, but whenever it is
applicable it gives
sharper results than Theorem~\ref{thm:lower_bd_via_Chebyshev}. 
In fact, Theorem~\ref{thm:small_d} shows that the factor $1-O(d^2/h^{1/2})$ 
in \eqref{eq:rnbnd1} 
can be omitted when $d \le 3$, giving 
$\Rbar(n) > (\pi e/2)^{-d/2}$.
Theorem~\ref{thm:small_d} is always applicable if
the Hadamard conjecture is true, since this conjecture
implies that $d\le 3$.

In \S\ref{sec:numerics}, we give some numerical examples
to illustrate Theorems~\ref{thm:lower_bd_via_Chebyshev} and~\ref{thm:small_d},
and to compare our results with previous bounds
on $\Dbar(n)$ and/or $\Rbar(n)$.

Rokicki \emph{et al}~\cite{Rokicki} showed, by extensive computation,
that $\Rbar(n) \ge 1/2$ for $n \le 120$, and conjectured that this
inequality always holds.
It seems difficult to bridge the gap between the constants $1/2$ and
$(\pi e/2)^{-3/2}$ by the probabilistic method.
The best that we can do is to improve the term of order $d^2/h^{1/2}$
in the bound~\eqref{eq:rnbnd1} at the expense of a more complicated
proof~-- for details see~\cite{rpb257}.

\section{The probabilistic construction}	\label{sec:construction}

We now describe our probabilistic construction
and prove some of its properties.
In the case $d=1$ our construction reduces to that of Best~\cite{Best}.

Let $A$ be a Hadamard matrix of order $h\ge 4$.  
We add a border of $d$ rows
and columns to give a larger (square) matrix $\widetilde{A}$ of order $n$. The 
border is defined by matrices $B$, $C$ and $D$ as shown:
\begin{equation}		\label{eq:block_matrix}
\widetilde{A} = \left[\begin{matrix} A & B\\ C & D\\
\end{matrix}\right]\,.
\end{equation} 
The $d\times d$ 
matrix $D - CA^{-1}B$ is known as the \emph{Schur complement} of
$A$ in $\widetilde{A}$ after Schur~\cite{Schur}.
The \emph{Schur complement lemma}
(see for example~\cite{Cottle}) gives
\begin{equation}
\det(\widetilde{A}) = \det(A)\det(D - CA^{-1}B).	\label{eq:SCL}
\end{equation}
In our construction the matrices
$A$, $B$, and $C$ have entries in $\{\pm1\}$.
We allow the matrix $D$ to have entries in $\{0, \pm 1\}$, but
each zero entry can be replaced by one of $+1$ or $-1$
without decreasing
$|\det(\widetilde{A})|$, so any lower bounds that we obtain on
$\max(|\det(\widetilde{A})|)$ are valid lower bounds on maximal 
determinants of $n\times n$ $\{\pm1\}$-matrices.
Note that the Schur complement is not in general a $\{\pm1\}$-matrix.

In the proof of Lemma~\ref{lemma:lower_bd_via_Chebyshev}
we show that our choice of $B$, $C$ and $D$ gives a Schur complement
$D - CA^{-1}B$ that, with positive probability, 
has sufficiently large determinant. 
{From} equation~\eqref{eq:SCL} and the fact that $A$ is a Hadamard matrix,
a large value of $\det(D-CA^{-1}B)$ implies a large
value of $\det(\widetilde{A})$.

\subsection{Details of the probabilistic construction}	\label{subsec:details}
Let $A$ be any Hadamard matrix of order $h$.
$B$ is allowed to range over the set 
of all $h\times d$ $\{\pm 1\}$-matrices, chosen uniformly
and independently from the $2^{hd}$ possibilities.
The $d \times h$ matrix 
$C = (c_{ij})$ is a function of $B$. We choose
\[c_{ij} = \sign (A^TB)_{ji}\,,\]
where  
\[\sign(x) := \begin{cases} +1 \text{ if } x \ge 0,\\
			   -1 \text{ if } x < 0.\\
	     \end{cases}
\]
To complete the construction,
we choose $D = -I$. As mentioned above, it is inconsequential that
$D$ is not a $\{\pm1\}$-matrix.

\subsection{Properties of the construction} 	\label{subsec:properties}

Define $F = CA^{-1}B$
and $G = 
F-D = F+I$ (so $-G$ is the Schur complement defined above).
Note that, since $A$ is a Hadamard matrix, $A^T = hA^{-1}$, so
$hF = CA^TB$. 

Since $B$ is random, we expect the elements of $A^TB$ to be usually of order
$h^{1/2}$.
The definition of $C$ ensures that there is no cancellation in the inner
products defining the diagonal entries of $hF = C\cdot (A^T B)$. Thus, we
expect the diagonal entries $f_{ii}$ of $F$ to be nonnegative and 
of order $h^{1/2}$,
but the off-diagonal entries $f_{ij}$ ($i\ne j$) 
to be of order unity with high probability.
Similarly for the elements of $G$.
This intuition is justified by
Lemmas~\ref{lemma:variance} and~\ref{lemma:sigma_asymptotics}.

In the following we denote the expectation of a random variable $X$
by $\E[X]$, and the variance by $\V[X] = \E[X^2] - \E[X]^2$.

Lemmas~\ref{lem:Efij}--\ref{lem:maxf} are essentially 
due to Best~\cite{Best}
and Lindsey.\footnote{See~\cite[footnote on pg.~88]{ES}.}
 
\begin{lemma}		\label{lem:Efij}
If $h \ge 2$ and $F=(f_{ij})$ is chosen as above, then
\[
\E[f_{ij}] = \begin{cases}
		\displaystyle 2^{-h}h\binom{h}{h/2}\; \text{ if }\; i = j,\\
			      0 \;\text{ if }\; i \ne j.\\
             \end{cases}
\]
\end{lemma}
\begin{proof}
The case $i=j$ follows as in Best~\cite[proof of Theorem~3]{Best}.
The case $i \ne j$ is easy, since $B$ is chosen randomly.
\end{proof}

\begin{lemma}                   \label{lem:maxf}
If $F = (f_{ij})$ is chosen as above, then $|f_{ij}| \le h^{1/2}$
for \hbox{$1 \le i, j \le d$}.
\end{lemma}
\begin{proof}
The matrix $Q := h^{-1/2}A^T$ is orthogonal with rows and columns of unit
length (in the Euclidean norm). Thus $||Qb||_2 = ||b||_2 = h^{1/2}$ for
each column $b$ of $B$.  Since $h^{1/2}F = C.QB$, each element
$h^{1/2}f_{ij}$ of $h^{1/2}F$
is the inner product of a row of $C$ (having length $h^{1/2}$) and a column
of $QB$ (also having length $h^{1/2}$).
It follows from the Cauchy-Schwartz
inequality that $|h^{1/2}f_{ij}| \le h^{1/2}\cdot h^{1/2} =  h$,
so $|f_{ij}| \le h^{1/2}$.
\end{proof}

\begin{lemma}           \label{lem:independence}
If $F$ is chosen as above and	
$\{i,j\}\cap\{k,\ell\} = \emptyset$,
then $f_{ij}$ and $f_{k\ell}$ are independent.
\end{lemma}   
\begin{proof}
This follows from the fact that
$f_{ij}$ depends only on the fixed matrix $A$ and on
columns $i$ and $j$ 
of $B$.
\end{proof}

\begin{lemma}	\label{lemma:u_sum}
Let $A \in \{\pm1\}^{h\times h}$ be a Hadamard matrix,
$C \in \{\pm1\}^{d\times h}$, and $U = CA^{-1}$.  
Then, for each $i$ with $1 \le i \le d$,
\[\sum_{j=1}^h u_{ij}^2 = 1.\]
\end{lemma}
\begin{proof}
Since $A$ is Hadamard,
$UU^T = 
h^{-1}CC^T$.
Also, since $c_{ij} = \pm 1$, ${\rm diag}(CC^T) = hI$.
Thus ${\rm diag}(UU^T) = I$.
\end{proof}

\begin{lemma}		\label{lem:fij2}
If $F = (f_{ij})$ is chosen as above, then
\begin{equation} 	\label{eq:Efij2}
\E[f_{ij}^2] = 1 \text{ for } \; i \ne j.
\end{equation}
\end{lemma}
\begin{proof} 
We can  assume, without loss of generality, that 
$i=1$, $j > 1$.  
Write $F = UB$, where $U = CA^{-1} = h^{-1}CA^T$.
Now 
\begin{equation}		\label{eq:sum_indep_vars}
f_{1j} = \sum_k u_{1k}b_{kj},
\end{equation}
where
\[u_{1k} = \frac{1}{h}\sum_{\ell} c_{1\ell}a_{k\ell},\;\;
  c_{1\ell} = {\rm sgn}\left(\sum_m b_{m1}a_{m\ell}\right).\]
Observe that $c_{1\ell}$ and $u_{1k}$ depend only on the first column of $B$.
Thus, $f_{1j}$ depends only on the first and $j$-th columns of $B$.
If we fix the first column of $B$ and take expectations over all choices
of the other columns, we obtain
\[\E[f_{1j}^2] = \E\left[\sum_k\sum_{\ell}u_{1k}u_{1\ell}b_{kj}b_{\ell j}
  \right].\]
The expectation of the terms with $k\ne \ell$ vanishes,
and the expectation of the terms with $k=\ell$
is $\sum_k u_{1k}^2$.
Thus, \eqref{eq:Efij2} follows from Lemma~\ref{lemma:u_sum}.
\end{proof}

\begin{lemma}	\label{lemma:variance}
Let $A$ be a Hadamard matrix of order $h \ge 4$ and $B$, $C$ be
$\{\pm1\}$-matrices chosen as above.  Let 
$G = F+I$ where $F = CA^{-1}B$.
Then 
\begin{eqnarray}
\E[g_{ii}] &=& 1 + \frac{h}{2^h}\binom{h}{h/2}\,,\label{eq:mu1}\\
\E[g_{ij}] &=& 0 \text{ for } 1 \le i, j \le d,\; i \ne j,\label{eq:mu2}\\
\V[g_{ii}] &=&
  1 + \frac{h(h-1)}{2^{h+1}}\binom{h/2}{h/4}^2 
   - \; \frac{h^2}{2^{2h}}\binom{h}{h/2}^2,\label{eq:sigma2}\\
\V[g_{ij}] &=& 1 \text{ for } 1 \le i, j \le d,\; i \ne j.\label{eq:sigmaij}
\end{eqnarray}
\end{lemma}
\begin{proof}
Since $G = F+I$,
the results~\eqref{eq:mu1}, \eqref{eq:mu2} and \eqref{eq:sigmaij}
follow from
Lemma~\ref{lem:Efij} and Lemma~\ref{lem:fij2} above.
Thus, we only need to prove~\eqref{eq:sigma2}.
Since $g_{ii} = f_{ii}+1$, 
it is sufficient to compute $\V[f_{ii}]$.

Since $A$ is a Hadamard matrix, $hF = CA^{T}B$. We compute the
second moment about the origin of the diagonal elements $hf_{ii}$
of $hF$.  Since $h$ is a Hadamard order and $h \ge 4$, we can write
$h = 4k$ where $k \in \Z$.  
Consider $h$ independent random variables $X_j \in \{\pm1\}$, $1 \le j \le h$,
where $X_j = +1$ with probability $1/2$.  Define random variables $S_1$,
$S_2$ by
\[S_1 = \sum_{j=1}^{4k} X_j, \;\;
  S_2 = \sum_{j=1}^{2k} X_j \; - \!\!\sum_{j=2k+1}^{4k} X_j\,.\]

Consider a particular choice of $X_1, \ldots, X_h$ and suppose that
$k+p$ of $X_1, \ldots, X_{2k}$ are $+1$,
and that $k+q$ of $X_{2k+1},\ldots,X_{4k}$ are $+1$.
Then we have $S_1 = 2(p+q)$ and $S_2 = 2(p-q)$.
Thus, taking expectations over all $2^{4k}$ possible (equally likely)
choices,
we see that
\begin{align*}
\E[|S_1 S_2|] = 4\E[|p^2-q^2|]
 &= \frac{4}{2^{4k}}\sum_p \sum_q \binom{2k}{k+p}\binom{2k}{k+q}|p^2-q^2|\\
 &= \frac{4}{2^{4k}} \cdot 2k^2\binom{2k}{k}^2
 = \frac{h^2}{2^{h+1}}\binom{2k}{k}^2\,.
\end{align*}
Here the closed form for the double sum is a special case 
of~\cite[Prop.~1.1]{rpb263}.
By the definitions of $B$, $C$ and $F$, we see that
$hf_{ii}$ is a sum of the form \hbox{$Y_1 + Y_2 + \cdots + Y_h$},
where each $Y_j$ is a random variable with the same distribution as $|S_1|$,
and each product $Y_j Y_\ell$ (for $j \ne \ell$) has the 
same distribution as $|S_1 S_2|$.
Also, $Y_j^2$ has the
same distribution as $|S_1|^2 = S_1^2$.  
The random variables $Y_j$ are not independent,
but by linearity of expectations we obtain
\[h^2\E[f_{ii}^2] = h\E[S_1^2] + h(h-1)\E[|S_1S_2|]
 = h^2 + h(h-1)\cdot \frac{h^2}{2^{h+1}}\binom{2k}{k}^2\,.\]
This gives
\[\E[f_{ii}^2] = 1 + \frac{h(h-1)}{2^{h+1}}\binom{2k}{k}^2\,.\]
The result for $\V[g_{ii}]$ now follows from
$\V[g_{ii}] = \V[f_{ii}] = \E[f_{ii}^2] - \E[f_{ii}]^2$.
\end{proof}

For convenience we write $\mu(h) := \E[g_{ii}] = \E[f_{ii}]+1$ 
and $\sigma(h)^2 := \V[g_{ii}]$.
If $h$ is understood from the context we write simply $\mu$
and $\sigma^2$ respectively.

To estimate $\mu$ and $\sigma^2$ from Lemma~\ref{lemma:variance},
we need a sufficiently
accurate estimate for a central binomial coefficient
$\binom{2m}{m}$ (where $m = h/2$ or $h/4$).
An asymptotic expansion 
for $\ln\binom{2m}{m}$ may be deduced from
Stirling's asymptotic expansion of $\ln\Gamma(z)$, as in~\cite{KS}. However,
\cite{KS} does not give an error bound.
We state such a bound in the following Lemma.

\begin{lemma}				\label{lemma:central_binomial_approx}
If $k$ and $m$ are positive integers, then
\begin{equation}		\label{eq:ln_binom_central}
\ln\binom{2m}{m} = m\ln 4 - \frac{\ln(\pi m)}{2}
        - \sum_{j=1}^{k-1} \frac{B_{2j}(1-4^{-j})}{j(2j-1)}\,m^{1-2j}
	+ e_k(m),
\end{equation}
where 
\begin{equation}		\label{eq:ln_binom_error}
|e_k(m)| < \frac{|B_{2k}|}{k(2k-1)}\,m^{1-2k}.
\end{equation}
\end{lemma}
\begin{proof}
Using the facts that $m$ is real and positive, and that the sign
of the Bernoulli number $B_{2k}$ is $(-1)^{k-1}$, we obtain 
from Olver~\cite[(4.03) and (4.05) of Ch.\ 8]{Olver74} that
\begin{equation}		\label{eq:Stirling_ln_Gamma}
\ln\Gamma(m) = \left(m-\scriptstyle\frac{1}{2}\right)\ln m - m
	+ \frac{\ln(2\pi)}{2}
	+ \sum_{j=1}^{k-1}\frac{B_{2j}}{2j(2j-1)}m^{1-2j}
	- (-1)^k r_k(m),
\end{equation}
where 
\begin{equation}		\label{eq:lnGamma_error_bd}
0 < r_k(m) < \frac{|B_{2k}|}{2k(2k-1)}m^{1-2k}.
\end{equation}
Now
\[\binom{2m}{m} = \frac{(2m)!}{m!m!}
	= \frac{2}{m}\frac{\Gamma(2m)}{\Gamma(m)^2}\,\raisecomma\] 
so from~\eqref{eq:Stirling_ln_Gamma} and the same equation with $m \mapsto 2m$
we obtain
\eqref{eq:ln_binom_central} with
\[
e_k(m) = (-1)^{k}(2r_k(m) - r_k(2m)). 
\]
Using the bound~\eqref{eq:lnGamma_error_bd}, this gives
\[
e_k(m) = \frac{(-1)^k|B_{2k}|}{k(2k-1)}\,m^{1-2k}\theta,
\]
where $-2^{-2k} < \theta < 1$.
In particular, $|\theta| < 1$, so we obtain the desired
bound~\eqref{eq:ln_binom_error}. 
\end{proof}

\begin{remark} 
{\rm
Lemma~\ref{lemma:central_binomial_approx} can be sharpened.  In fact,
$e_k(m)$ has the same sign as the first omitted term (corresponding
to $j=k$) and has smaller magnitude.  This is
proved in~\cite[Corollary~2]{rpb267}.
}
\end{remark}

We now show that $\mu(h)$ is of order $h^{1/2}$, and that $\sigma(h)$
is bounded.

\pagebreak[3]

\begin{lemma}	\label{lemma:sigma_asymptotics}
For $h \in 4\Z$, $h \ge 4$,
we have
\begin{equation}	\label{eq:sigma_weak_upper_bound}
\sigma(h)^2 < 1
\end{equation}
and
\begin{equation}	\label{eq:ineq_mu1}
\sqrt{\frac{2h}{\pi}} + 0.9 < \mu(h) < \sqrt{\frac{2h}{\pi}} + 1.
\end{equation}
\end{lemma}
\begin{proof}
{From} Lemma~\ref{lemma:central_binomial_approx} with $k=2$ and $m$ a
positive integer, we have
\begin{equation}				\label{eq:centralk2}
\binom{2m}{m} = \frac{4^m}{\sqrt{\pi m}}
		   \exp\left[-\frac{1}{8m} + \frac{\theta_m}{180m^3}\right]
		   \,,
\end{equation}
where $|\theta_m| < 1$.

First consider the bounds~\eqref{eq:centralk2} on $\mu(h)$.
Taking $m=h/2$ and using the expression~\eqref{eq:mu1} for $\mu(h)$,
the inequality~\eqref{eq:ineq_mu1} is equivalent to
\[
\sqrt{\frac{m}{\pi}} - \frac{1}{20} < \frac{m}{4^m}\binom{2m}{m} <
\sqrt{\frac{m}{\pi}}\,\raisedot
\]
The upper bound is immediate from~\eqref{eq:centralk2},
since $-\frac{1}{8m} + \frac{1}{180m^3} < 0$.

For the lower bound, a computation verifies the inequality
for $m=2$, since $\sqrt{{2}/{\pi}}-\frac{1}{20} <
\frac{3}{4} = \frac{m}{4^m}\binom{2m}{m}$. 
Hence, we can assume that $m \ge 4$.  The lower bound now follows
from~\eqref{eq:centralk2}, since
\[
\frac{m}{4^m}\binom{2m}{m} > \sqrt{\frac{m}{\pi}}\,
\exp\left[-\frac{1}{8m} - \frac{1}{180m^3}\right]
> \sqrt{\frac{m}{\pi}}\,\left[1-\frac{1}{8m} - \frac{1}{180m^3}\right]
\]
and 
\[
\sqrt{\frac{m}{\pi}}\left[\frac{1}{8m} + \frac{1}{180m^3}\right] <
\frac{1}{20}\,\raisedot
\]

Now consider the upper bound~\eqref{eq:sigma_weak_upper_bound}
on $\sigma(h)^2$.  From~\eqref{eq:centralk2} we have
\[\binom{h/2}{h/4}^2 < \frac{2^{h+2}}{\pi h}
			\exp\left[-\frac{1}{h}+\frac{32}{45h^3}\right]\]
and
\[\binom{h}{h/2}^2 > \frac{2^{2h+1}}{\pi h}
                        \exp\left[-\frac{1}{2h}-\frac{4}{45h^3}\right]\,.\]
Using these inequalities in~\eqref{eq:sigma2}
and simplifying gives
\begin{align}		\label{eq:sigma_upper_bound}
\sigma(h)^2 < 1 &+ \frac{2h}{\pi}\left[
	\exp\left(-\frac{1}{h}+\frac{32}{45h^3}\right) -
   		  \exp\left(-\frac{1}{2h}-\frac{4}{45h^3}\right)\right]
		\nonumber\\
	&- \frac{2}{\pi}\exp\left(-\frac{1}{h}+\frac{32}{45h^3}\right)\,.
\end{align}
It is easy to see that the term in square brackets is negative for $h \ge
4$, so~\eqref{eq:sigma_upper_bound}
implies~\eqref{eq:sigma_weak_upper_bound}.
\end{proof}
\begin{remark}		\label{remark:sharper_sigma_bounds}
We can show from~\eqref{eq:sigma_upper_bound} and
a corresponding lower bound on $\sigma(h)^2$
that $\sigma(h+4)^2 < \sigma(h)^2$,
so $\sigma(h)^2$ is monotonic decreasing and bounded above
by $\sigma(4)^2 = \frac{1}{4}$.
Also, for large~$h$ we have
$\sigma(h)^2 = (1 - 3/\pi) + O(1/h)$.  Since these results are not
needed below, we omit the details.
\end{remark}

\section{A probabilistic lower bound} \label{sec:bounds}

We now prove
lower bounds on $\Dbar(n)$ and $\Rbar(n)$
where, as usual, $n = h+d$ and $h$ is the order of a Hadamard matrix.
The key result is Lemma~\ref{lemma:lower_bd_via_Chebyshev}.
Theorem~\ref{thm:lower_bd_via_Chebyshev} simply converts the result
of Lemma~\ref{lemma:lower_bd_via_Chebyshev} into
lower bounds on $\Dbar(n)$ and $\Rbar(n)$, giving away a little
for the sake of simplicity in the latter case.

For the proof
of Lemma~\ref{lemma:lower_bd_via_Chebyshev} we need the following
bound on the determinant of a matrix which is ``close'' to
the identity matrix.
It is due to Ostrowski~\cite[eqn.~(5,5)]{Ostrowski38};
see also~\cite[Corollary~1]{rpb258}. 

\begin{lemma}[Ostrowski]	\label{lemma:optimal_pert_bound}
If $M = I - E \in \R^{d\times d}$, 
$|e_{ij}| \le \ve$ for $1 \le i, j \le d$,
and $d\ve \le 1$, then 
\[\det(M) \ge 1 - d\ve.\]
\end{lemma}

The idea of Lemma~\ref{lemma:lower_bd_via_Chebyshev} 
is that we can, with positive probability, 
apply Lemma~\ref{lemma:optimal_pert_bound}
to the matrix $M = \mu^{-1}G$, thus obtaining a lower
bound on the maximum value attained by $\det(G)$.
\begin{lemma}		\label{lemma:lower_bd_via_Chebyshev}
Suppose $d \ge 1$, $4 \le h \in {\Had}$, $n = h+d$, $G$ as
in~$\S\ref{subsec:properties}$. Then, with positive probability,
\begin{equation}
\frac {\det G}{\mu^d} \ge 1 - \frac{d^2}{\mu}\,.	\label{eq:fudge1}
\end{equation}
\end{lemma}
\begin{proof}
Let $\lambda$ be a positive parameter to be chosen later, and $\mu = \mu(h)$.
We say that $G$ is \emph{good} if the conditions of
Lemma~\ref{lemma:optimal_pert_bound}
apply with $M = \mu^{-1}G$ and $\ve = \lambda/\mu$.
Otherwise $G$ is \emph{bad}.

Assume $1 \le i, j \le d$. 
{From} Lemma~\ref{lemma:variance}, $\V[g_{ij}] = 1$ for $i \ne j$;
from Lemma~\ref{lemma:sigma_asymptotics},
$\V[g_{ii}] = \sigma^2 < 1$. 
It follows from Chebyshev's
inequality~\cite{Chebyshev}
that
\[\Prob[|g_{ij}| \ge \lambda] \le \frac{1}{\lambda^2}
\;\text{ for }\; i\ne j,\]
and
\[\Prob[|g_{ii} - \mu| \ge \lambda] \le
 \frac{\sigma^2}{\lambda^2}\,\raisedot\]
Thus,
\[\Prob[G \text{ is bad}] \le \frac{d(d-1)}{\lambda^2}
	+ \frac{d\sigma^2}{\lambda^2}
 	< \frac{d^2}{\lambda^2}
\,\raisedot
\]

Taking $\lambda = d$ gives
$\Prob[G \text{ is bad}] < 1$, 
so $\Prob[G \text{ is good}]$ is positive.
Whenever $G$ is good we can apply Lemma~\ref{lemma:optimal_pert_bound}
to $\mu^{-1}G$, obtaining
$\mu^{-d}\det(G) = 
\det(\mu^{-1}G) \ge 1 - d\ve = 
1 - d\lambda/\mu = 1 - d^2/\mu$.
\end{proof}

The following lemma is useful for deducing lower bounds on $\Rbar(n)$.
\begin{lemma} \label{lemma:uncond2}
If $n = h+d > h > 0$, then
\[(h/n)^n > \exp(-d - d^2/h).\]
\end{lemma}
\begin{proof}
Writing $x = d/n$,
the inequality
$\ln(1-x) > -x/(1-x)$
implies that
\[(1-x)^n \,>\, \exp\left(-\frac{nx}{1-x}\right)\,.\]
Since $1-x = h/n$, we obtain
\[\left(\frac{h}{n}\right)^n
 \,>\, \exp\left(\frac{-d}{1-d/n}\right)
  = \exp(-d - d^2/h).\\[-30pt]\]	
\end{proof}
\pagebreak[3]

We are now ready to prove our main result.
Theorem~\ref{thm:lower_bd_via_Chebyshev} gives lower bounds on
$\Dbar(n)$ and ${\cal R}(n)$.  If the reader needs a lower bound
for a specific value of $n$, then the inequality~\eqref{eq:Dbd1}
should be used.  The inequality~\eqref{eq:fudge2} is slightly weaker
than what can be obtained simply by dividing both sides
of~\eqref{eq:Dbd1} by $n^{n/2}$, but it shows more clearly the asymptotic
behaviour if~$n$ and~$h$ are large but $d$ is small.
 
\begin{theorem} \label{thm:lower_bd_via_Chebyshev}
Suppose $d \ge 1$, $4 \le h \in {\Had}$, and $n = h+d$.
Then
\begin{equation}				\label{eq:Dbd1}
\Dbar(n) \ge h^{h/2}\mu^d(1 - d^2/\mu),
\end{equation}
where $\mu = 1 + \frac{h}{2^h}\binom{h}{h/2}$.
Also,
\begin{equation}
{\cal R}(n) > \left(\frac{2}{\pi e}\right)^{d/2}
	\left(1 - d^2\sqrt{\frac{\pi}{2h}}\right)\,. \label{eq:fudge2}
\end{equation}
\end{theorem}
\begin{proof}
Lemma~\ref{lemma:lower_bd_via_Chebyshev} and the Schur complement lemma
imply that there exists an
$n \times n$ $\{\pm1\}$-matrix with determinant
at least \hbox{$h^{h/2}\mu^d(1-d^2/\mu)$}.
Thus, \eqref{eq:Dbd1} follows from the definition of~$\Dbar(n)$.

We now show that~\eqref{eq:fudge2} follows from~\eqref{eq:Dbd1} by some
elementary inequalities.
Write $c := \sqrt{2/\pi}$.  
We can assume that $d^2 < ch^{1/2}$, for there is nothing to prove unless
the right side of~\eqref{eq:fudge2} is positive.
{From} Lemma~\ref{lemma:sigma_asymptotics},
$ch^{1/2} < \mu$, so $d^2 < \mu$.
Also, from~\eqref{eq:Dbd1},
\begin{equation}				\label{eq:17A}
\Rbar(n) \ge \frac{h^{h/2} \mu^d}{n^{n/2}}
 \left(1 - \frac{d^2}{\mu}\right)\,\raisedot
\end{equation}
Using  $ch^{1/2} < \mu$, this gives
\[\Rbar(n) > c^d (h/n)^{n/2}(1 - d^2/\mu).\]
By Lemma~\ref{lemma:uncond2},
$(h/n)^n > \exp(-d-d^2/h)$, so
\begin{equation}	\label{eq:C1bd1}
\Rbar(n) > c^d e^{-d/2} f
 = \left(\frac{2}{\pi e}\right)^{d/2}\!\!f,
\end{equation}
where 
\begin{equation}	\label{eq:C1f}
f = \exp\left(-\,\frac{d^2}{2h}\right) \left(1 - \frac{d^2}{\mu}\right).
\end{equation}
Thus, to prove~\eqref{eq:fudge2}, it suffices to prove that
$f \ge 1 - {d^2}/{(ch^{1/2})}$.
Since $\exp(-d^2/(2h)) \ge 1 - d^2/(2h)$,
it suffices to prove that
\begin{equation}	\label{eq:ineq_dhmu}
 \left(1 - \frac{d^2}{2h}\right)
 \left(1 - \frac{d^2}{\mu}\right)
 \ge 1 - \frac{d^2}{ch^{1/2}}\,\raisedot
\end{equation}
Expanding and simplifying shows that the inequality~\eqref{eq:ineq_dhmu}
is equivalent to
\begin{equation}	\label{eq:ineq_dhmu2}
2h + \mu \le d^2 + \mu\sqrt{2\pi h}.
\end{equation}
Now, by Lemma~\ref{lemma:sigma_asymptotics}, 
$\mu > c\sqrt{h} + 0.9$, so
$\mu\sqrt{2\pi h} > 2h + 0.9\sqrt{2\pi h}$
(using  $c\sqrt{2\pi} = 2$).
Thus, to prove~\eqref{eq:ineq_dhmu2}, it suffices to show that
$\mu \le d^2 + 0.9\sqrt{2\pi h}$. 
Using Lemma~\ref{lemma:sigma_asymptotics} again, we have
$\mu \le ch^{1/2} + 1$, so it suffices to show that
\[ch^{1/2} + 1 \le 0.9\sqrt{2\pi h} + d^2.\]
This follows from $c \le 0.9\sqrt{2\pi}$ and $1 \le d^2$,
so the proof is complete.
\end{proof}

\begin{remark}	\label{remark:lower_bd_via_Chebyshev}
{\rm
The inequality~\eqref{eq:fudge2} of
Theorem~\ref{thm:lower_bd_via_Chebyshev} gives a nontrivial lower bound
on ${\cal R}(n)$ iff the second factor in the bound is positive, 
i.e.~iff $h > {\pi}d^4/2$.
By Livinskyi's results~\cite{Livinskyi}, this condition 
holds for all sufficiently large $n$ 
(assuming as always that we choose the maximal $h\le n$ for given $n$).
}
\end{remark}

The Hadamard conjecture implies that $d\le 3$. Theorem~\ref{thm:small_d}
improves on Theorem~\ref{thm:lower_bd_via_Chebyshev} under the
assumption that $d \le 3$.
The proof of Theorem~\ref{thm:small_d} is conceptually
simpler than that of Theorem~\ref{thm:lower_bd_via_Chebyshev}, 
since it does not require any bounds on the variance $\sigma(h)^2$.
In the proof of Theorem~\ref{thm:small_d} we simply expand $\det(G)$,
obtaining $d!$ terms.  By Lemma~\ref{lem:independence}, the expectation
of the diagonal term is $\E[g_{11}\cdots g_{dd}] = \mu^d$.
The expectation of the off-diagonal terms can be bounded to
give the desired lower bound on $\Dbar(n)$. The same approach
gives weak results for $d > 3$ because of the large number ($d!-1$) of
off-diagonal terms (see~\cite[Theorem~1]{rpb253}).

\begin{theorem} \label{thm:small_d}
If $1\le d\le 3$, $h \in {\Had}$, $n = h+d$, and $\mu $ as
in~\eqref{eq:Dbd1}, then
\[
\Dbar(n) \ge h^{h/2}(\mu^d - \eta)
\;\text{ and }\;
\Rbar(n) > \left(\frac{2}{\pi e}\right)^{d/2},
\]
where\\[-25pt]
\[
\eta = 
\begin{cases}
        d-1 \text{ if } 1 \le d \le 2,\\
        5h^{1/2} + 3 \text{ if } d = 3.\\
\end{cases}
\]
\end{theorem}
\begin{proof}
It is easy to verify the result for $h \in \{1,2\}$, so suppose that $h \ge 4$.
For notational convenience we give the proof for the case $d=2$.  
The cases $d\in\{1,3\}$ are
similar.\footnote{A detailed proof for the case $d=3$ is given
in~\cite[proof of Lemma~17]{rpb257}.}

Since $G = F+I$, we have $g_{ii} = f_{ii}+1$ and
$\det(G) = g_{11}g_{22} - f_{12}f_{21}$. 
By Lemma~\ref{lem:independence}, the diagonal elements $g_{11}$ and $g_{22}$
are independent, so
\[\E[g_{11}g_{22}] = \E[g_{11}]\E[g_{22}] = \mu^2.\]
By the Cauchy-Schwarz inequality and Lemma~\ref{lem:fij2},
\[\E[f_{12}f_{21}]^2 \le \E[f_{12}^2]\E[f_{21}^2] = 1.\]
Thus
\[\E[\det(G)] = \E[g_{11}g_{22}] - \E[f_{12}f_{21}] \ge \mu^2-1.\]
There must exist some $G_0$ with $\det(G_0) \ge \E[\det(G)] \ge \mu^2-1$;
hence \[\Dbar(n) \ge h^{h/2}(\mu^2-1).\]
This proves the required lower bound for $\Dbar(n)$ if $d=2$.
We now deduce the required lower bound for $\Rbar(n) = \Dbar(n)/n^{n/2}$.
Define $c := \sqrt{2/\pi}$ and $K := 0.9/c$.
{From} Lemma~\ref{lemma:sigma_asymptotics},
$\mu \ge c(h^{1/2} + K)$,
so $\mu^2 \ge c^2 h(1 + 2Kh^{-1/2})$.
Thus, using $n=h+2$,
\[\Dbar(n) \ge c^2 h^{n/2}\left(1 + 2Kh^{-1/2} - \frac{\eta}{c^2 h}\right).\]
{From} Lemma~\ref{lemma:uncond2} with $d=2$, 
$(h/n)^{n/2} \ge 
e^{-1 -2/h} \ge e^{-1}(1-2/h)$, so
\[\Rbar(n) = \frac{\Dbar(n)}{n^{n/2}}
 \ge \left(\frac{2}{\pi e}\right)
	\left(1 + 2Kh^{-1/2} - \frac{1}{c^2 h}\right)
        \left(1-\frac{2}{h}\right).\]
Since $K$ is positive,
the term $2Kh^{-1/2}$ dominates the $O(h^{-1})$ terms, and
the result $\Rbar(n) > 2/(\pi e)$ follows for all sufficiently large~$h$.
In fact, a small computation shows that the
inequality holds for all $h \ge 4$.
\end{proof}

\section{Numerical examples}			\label{sec:numerics}

In this section we give some numerical comparisons between our
lower bounds and previously-known bounds.

There are two well-known approaches to constructing a large-determinant
$\{\pm1\}$-matrix of order~$n$.  The \emph{bordering} approach 
takes a Hadamard matrix $H$ of order $h\le n$ and adjoins a border of
$d=n-h$ rows and columns. The border is constructed in a manner intended to
result in a large determinant. 
Previously, deterministic constructions were used~--
see for example~\cite[Lemma~7]{rpb249}.
In this paper we have used a probabilistic construction.

The \emph{minors} approach
takes a Hadamard matrix $H_{+}$
of order $h_{+}\ge n$ and finds an $n\times n$ submatrix with large determinant.
This approach was used deterministically
by Koukouvinos \emph{et~al}~\cite{KMS00,KMS01}, 
and probabilistically by
de Launey and Levin~\cite{LL}.
The deterministic approach can
be generalised using a theorem of Sz\"oll\H{o}zi~\cite{Szollosi10},
and this is better for $h_{+} \le n+6$ 
than the probabilistic approach
of~\cite{LL}~-- see~\cite[Remarks~6 and 22]{rpb249}.  

To illustrate Theorem~\ref{thm:lower_bd_via_Chebyshev},
consider the case $n=668$, $d=4$.
At the time of writing, $n$ is the smallest positive multiple of $4$ that
is not known to be in~$\Had$.  It is known that $h := n-4 \in\Had$
and $h_{+} := n+4\in\Had$.

\noindent
The deterministic bordering approach~\cite[Lemma~7]{rpb249} gives
a lower bound 
$\Rbar(n) \ge 2^d h^{h/2}/n^{n/2} \approx 4.88\times 10^{-6}$.
The deterministic minors approach 
gives a lower bound
$\Rbar(n) \ge 16 h_{+}^{h_{+}/2-4}/n^{n/2} \approx 2.60\times 10^{-4}$.
The probabilistic bordering approach 
of
Theorem~\ref{thm:lower_bd_via_Chebyshev} gives a lower bound
(eqn.~\eqref{eq:17A} above)
$\Rbar(n) \ge h^{h/2}\mu^d(1-d^2/\mu)/n^{n/2}
 \approx 1.69\times 10^{-2}$,
where $\mu$ is as in~\eqref{eq:Dbd1}. 
For comparison, our
conjectured lower bound is $(\pi e/2)^{-d/2} \approx 5.48\times 10^{-2}$.

\begin{table}[ht]        
\begin{center}          
\caption{Asymptotics of lower bounds on $\Rbar(n)$ as $n\to\infty$.}
\vspace*{10pt}
\begin{tabular}{|c|c|c|c|}
\hline
$d$ 	& KMS~\cite{KMS00} 
	& B\&O~\cite{rpb249}
        & Theorem~\ref{thm:small_d} \\
\hline
& & & \\[-8pt]
$1$     & $\displaystyle 4\left(\frac{e}{n}\right)^{3/2}
	  \approx \frac{17.93}{n^{3/2}}$
        & $\displaystyle 
	   \;\;
	   \left(\frac{2}{\pi e}\right)^{1/2} 
	   \approx 0.4839$
	& $\displaystyle \left(\frac{2}{\pi e}\right)^{1/2} 
	   \approx 0.4839$ \\
& & & \\[-8pt]
$2$     & $\displaystyle
	   \;\;\;\;
	   \frac{2e}{n}
	   \;\;\;\;\ 
	   \approx \frac{5.437}{n}$
        & $\displaystyle \left(\frac{8}{\pi e^2 n}\right)^{1/2} 
	   \approx \frac{0.5871}{n^{1/2}}$
	& $\displaystyle \;\; \frac{2}{\pi e}
	   \;\;\;\;\;\;
	   \approx 0.2342$\\
& & & \\[-8pt]
$3$     & $\displaystyle
	   \;\;
	   \left(\frac{e}{n}\right)^{1/2}
	   \approx \frac{1.649}{n^{1/2}}$
        & $\displaystyle
	   \;\;
	   \left(\frac{e}{n}\right)^{1/2} 
	   \;\;
	   \approx \frac{1.649}{n^{1/2}}$
	& $\displaystyle \left(\frac{2}{\pi e}\right)^{3/2} 
	   \approx 0.1133$\\[10pt]
\hline
\end{tabular}
\end{center}
\end{table}

\pagebreak[3]
To illustrate Theorem~\ref{thm:small_d},
Table~{\Tablea} summarises the asymptotics of some lower bounds
on $\Rbar(n)$ for
$d = (n \bmod 4) \in \{1,2,3\}$, assuming that $n-d\in\Had$, $n+4-d\in\Had$.
The bounds are those given in Koukouvinos \emph{et al}~\cite{KMS00},
Brent and Osborn~\cite[Table~1]{rpb249},
and Theorem 
\ref{thm:small_d} of the present paper.
It can be seen that we improve on the previous bounds by a factor
of order at least $n^{1/2}$ for $d \in \{2,3\}$.

\pagebreak[3]
Since asymptotics may be misleading for small~$n$,
Table {\Tableb} gives lower bounds on $\Rbar(n)$ for various
values of $n \equiv 2 \bmod 4$ (so $d=2$).

\begin{table}[ht]        
\begin{center}          
\caption{Comparison of lower bounds on $\Rbar(n)$ for $d=2$.}
\vspace*{10pt}
\begin{tabular}{|c|c|c|c|c|}
\hline
$n$ 	& KMS~\cite{KMS00} 
	& B\&O~\cite{rpb249}
	& Thm.~\ref{thm:lower_bd_via_Chebyshev}
        & Thm.~\ref{thm:small_d} \\
\hline
10 & 0.4147 & 0.1856 & -- & 0.3752\\
14 & 0.3183 & 0.1569 & -- & 0.3609\\
18 & 0.2581 & 0.1384 & 0.0127 & 0.3498\\
98 & 0.0538 & 0.0593 & 0.1601 & 0.2897\\
998 & 0.0054 & 0.0186 & 0.2142 & 0.2524\\
limit & 0.0000 & 0.0000 & 0.2342 & 0.2342\\
\hline
\end{tabular}
\end{center}
\end{table}

In the case $d=3$, 
a computation shows that
the first bound of our Theorem~\ref{thm:small_d}
is sharper than the bound $\Dbar(n) \ge (n+1)^{(n-1)/2}$ of~\cite[Thm.~2]{KMS00} 
if $n \ge 135$ (where the latter bound assumes that $n+1\in\Had$).

\section*{Acknowledgements}
We thank Robert Craigen for informing us of the work of 
his student Ivan
Livinskyi~\cite{Livinskyi}, and Will Orrick for a copy of the
unpublished report~\cite{Rokicki}.
We also thank an anonymous referee who helped us to improve
both the rigor 
and clarity of the paper.

The first author was supported in part
by Australian Research Council grant DP140101417.


\begin{thebibliography}{99}

\bibitem{AS}
N. Alon and J. H. Spencer, 
\emph{The Probabilistic Method}, 3rd edn., Wiley, 2008.


\bibitem{Best}
M. R. Best, The excess of a Hadamard matrix,
\emph{Nederl.\ Akad.\ Wetensch.\ Proc.\ Ser. A} \textbf{80}
$=$ 
\emph{Indag.\ Math.\ }\textbf{39} (1977), 
357--361.

\bibitem{rpb267}
R. P. Brent,
\emph{Asymptotic approximation of central binomial coefficients with
rigorous error bounds}, arXiv:1608.04834v1, 17 August 2016, 11~pp.

\bibitem{rpb263}
R. P. Brent, C. Krattenthaler and S. O. Warnaar,
Discrete analogues of Mehta-type integrals,
{\em J.\ Combin.\ Theory Ser.\ A} \textbf{144} (2016), 80--138.
\url{http://dx.doi.org/10.1016/j.jcta.2016.06.005}.

\bibitem{rpb249}
R. P. Brent and J. H. Osborn,
{General lower bounds on maximal determinants of binary matrices},
\emph{Electron.\ J.\ Comb.\ }
\textbf{20}(2), 2013, \#P15, 
12~pp.


\bibitem{rpb253} 
R. P. Brent, J. H. Osborn and W. D. Smith,
\emph{Lower bounds on maximal determinants of $\pm1$ matrices via the
probabilistic method}, arXiv:1211.3248v3, 5~May\ 2013, 32~pp.

\bibitem{rpb257} 
R. P. Brent, J. H. Osborn and W. D. Smith,
\emph{Lower bounds on maximal determinants of binary matrices via the
probabilistic method}, arXiv:1402.6817v2, 13~March\ 2014, 37~pp.

\bibitem{rpb258}
R. P. Brent, J. H. Osborn and W. D. Smith,
Bounds on determinants of perturbed diagonal matrices,
\emph{Linear Alg.\ Appl.\ }\textbf{466} (2015), 21--26.


\bibitem{BS}
T. A. Brown and J. H. Spencer,
Minimization of $\pm1$ matrices under line shifts,
\emph{Colloq.\ Math.} 
\textbf{23} (1971), 165--171.
Erratum \emph{ibid} pg.~177.

\bibitem{Chebyshev}
P. Chebyshev,
Des valeurs moyennes,
\emph{J.\ Math.\ Pure Appl.\ }%
\textbf{2} 
(1867), 177--184.

\bibitem{CL65}
G. F. Clements and B. Lindstr\"om,
A sequence of $(\pm 1)$-determinants with large values,
\emph{Proc.\ Amer.\ Math.\ Soc.\ }\textbf{16} (1965), 548--550.



\bibitem{Cottle}
R. W. Cottle, 
Manifestations of the Schur complement,
\emph{Linear Alg.\ Appl.} \textbf{8} (1974), 189--211.













\bibitem{ES}
P. Erd\H{o}s and J. Spencer,
\emph{Probabilistic Methods in Combinatorics},
Akad\'emiai Kiad\'o, Budapest, 1974.
Also published by Academic Press, New York, 1974.

\bibitem{FK}
N. Farmakis and S. Kounias, 
The excess of Hadamard matrices and optimal designs,
\emph{Discrete Math.\ }\textbf{67} (1987), 165--176.

\bibitem{Hadamard}
J. Hadamard,
R\'esolution d'une question relative aux d\'eterminants,
\emph{Bull. des Sci. Math.} \textbf{17} (1893), 240--246.






\bibitem{KS}
D. Kessler and J. Schiff, 
\emph{The asymptotics of factorials, binomial coefficients
and Catalan numbers},
\url{http://u.math.biu.ac.il/~schiff/Papers/prepap3.pdf},
April 2006.

\bibitem{KMS00}
C. Koukouvinos, M. Mitrouli and J. Seberry,
Bounds on the maximum determinant for $(1,-1)$ matrices,
\emph{Bull.\ Inst.\ Combinatorics and its Applications}
\textbf{29} (2000), 39--48.

\bibitem{KMS01}
C. Koukouvinos, M. Mitrouli and J. Seberry,
An algorithm to find formul{\ae} and values of minors for Hadamard matrices,
\emph{Linear Alg.\ Appl.\ }\textbf{330} (2001), 129--147.

\bibitem{LL}
W. de Launey and D. A. Levin, 
$(1,-1)$-matrices with near-extremal properties,
\emph{SIAM J.\ Discrete Math.\ }\textbf{23} (2009), 1422--1440.

\bibitem{Livinskyi}
I. Livinskyi, 
\emph{Asymptotic existence of Hadamard matrices},
M.Sc.\ thesis, University of Manitoba,
\url{http://hdl.handle.net/1993/8915}, 2012.





\bibitem{Olver74}
F.\ W.\ J.\ Olver,
\emph{Asymptotics and Special Functions},
Academic Press, New York, 1974.



\bibitem{Ostrowski38}   
A. M. Ostrowski,
Sur l'approximation du d\'eterminant de Fredholm par les
d\'eterminants des syst\`emes d'equations lin\'eaires,
\emph{Ark.\ Math.\ Stockholm} \textbf{26A} (1938), 1--15.



\bibitem{Rokicki}
T. Rokicki, I. Kazmenko, J-C. Meyrignac, W. P. Orrick, V. Trofimov and
J. Wroblewski,
\emph{Large determinant binary matrices: 
results from Lars Backstrom's programming contest},
unpublished report, July 31, 2010.


\bibitem{Schur}
I. Schur, 
\"Uber Potenzreihen, die im Innern des Einheitskreises besch\-r\"ankt sind,
\emph{J.\ Reine Angew.\ Math.} 
\textbf{147} (1917), 205--232.





\bibitem{Szollosi10}
F. Sz\"oll\H{o}si, 
Exotic complex Hadamard matrices and their equivalence,
\emph{Cryptogr.\ Commun.\ }%
\textbf{2} (2010), 187--198.







\end{thebibliography}
\end{document}